\documentclass[12pt,letterpaper]{article}
\usepackage{amsfonts,amssymb,latexsym,amsthm}

\newtheorem{theorem}{Theorem}
\newtheorem{lemma}[theorem]{Lemma}
\newtheorem{corollary}[theorem]{Corollary}
\newtheorem{question}[theorem]{Question}
\newtheorem{example}[theorem]{Example}
\newtheorem*{remark}{Remark}

\begin{document}

\author{Ramiro de la Vega \footnote{Universidad de los Andes, Bogot\'a,
Colombia, \ \ rade@uniandes.edu.co}}
\title{On the reflection of the countable chain condition
}

\maketitle

\begin{abstract}
We study the question of when an uncountable ccc topological space
$X$ contains a ccc subspace of size $\aleph_1$. We show that it
does if $X$ is compact Hausdorff and more generally if $X$ is
Hausdorff with $\mathrm{pct}(X) \leq \aleph_1$. For each regular
cardinal $\kappa$, an example is constructed of a ccc Tychonoff
space of size $\kappa$ and countable pseudocharacter but with no
ccc subspace of size less than $\kappa$. We also give a ccc
compact $T_1$ space of size $\kappa$ with no ccc subspace of size
less than $\kappa$.
\end{abstract}

A topological space $(X,\tau)$ has \emph{the countable chain
condition} ($X$ is \emph{ccc}) if for any uncountable $\mathcal{F}
\subseteq \tau$ there are distinct $U,V \in \mathcal{F}$ such that
$U \cap V \neq \emptyset$. We are interested in the following

\begin{question}\label{reflection}
Does any uncountable ccc topological space contain an uncountable
ccc subspace of size $\aleph_1$?
\end{question}

The corresponding questions for \emph{second-countability} (in
place of ccc) and for  \emph{separability} have trivial
affirmative answers. The question for \emph{Lindel\"ofness} is
highly non-trivial and has recieved substantial attention (see,
for instance, \cite{BT}, \cite{HJ}, \cite{KT} and \cite{S}). It
was claimed in the first paragraph of \cite{BT} that the answer to
Question \ref{reflection} is affirmative ``by a standard easy
L\"owenheim-Skolem argument." It turns out that this is not the
case as the following simple example shows.

\begin{example}\label{hausdorffexample}
For each infinite cardinal $\kappa$, there is a Hausdorff ccc
space with no uncountable ccc subspace of size less than
$2^\kappa$.
\end{example}
\begin{proof}
Let $X=2^\kappa$ with the topology $\tau$ generated by sets of the
form $P \cap C$ where $P$ is basic clopen in the product topology
and $C \subseteq X$ with $|X \setminus C|<2^\kappa$. Note that the
generating set is closed under finite intersections so it is in
fact a base for $\tau$.

Clearly $\tau$ contains the product topology on $X$ so $(X,\tau)$
is a Hausdorff space. Moreover since any clopen in the product
topology has size $2^\kappa$ we have that for any two basic open
sets $P_1 \cap C_1$ and $P_2 \cap C_2$, $$(P_1 \cap C_1) \cap (P_2
\cap C_2) \neq \emptyset \iff P_1 \cap P_2 \neq \emptyset.$$ It
follows that $(X,\tau)$ is ccc (since the product topology is
ccc).

Finally if $Y$ is a subspace of $X$ with $|Y|<2^\kappa$ then $Y$
is discrete. To see this just note that given $y \in Y$ we can
take $C=(X \setminus Y) \cup \{y\}$ and since $|X \setminus C|=|Y
\setminus \{y\}|<2^\kappa$, we have that $\{y\}=Y \cap (X \cap C)$
is open in $Y$. Hence $Y$ is not ccc unless $Y$ is countable.
\end{proof}

A standard way to approach Question \ref{reflection} would be to
take an elementary submodel $M$ of (a large enough fragment of)
the universe with $\omega_1 \cup \{X,\tau\} \subseteq M$ and
$|M|=\aleph_1$, and then try to prove that the subspace $X \cap M$
is ccc (see \cite{dow} for more information on elementary
submodels). By elementarity one can easily see that $X\cap M$ with
the topology $\tau_M$ generated by $\{U \cap M : U \in \tau \cap M
\}$ is ccc. However $\tau_M$ is usually strictly coarser than the
subspace topology on $X\cap M$ (see \cite{JT} for many examples of
this phenomenon), so this gives us a positive answer only in cases
where we can guarantee that the two topologies coincide. It is
well known and easy to see by elementarity that this is the case
whenever $\chi(X) \subseteq M$, so we have

\begin{theorem}\label{firstcountable}
Any uncountable ccc space of character at most $\aleph_1$ contains
a ccc subspace of size $\aleph_1$.
\end{theorem}

Since ccc linearly ordered spaces are first countable we get

\begin{corollary}
Every uncountable ccc linearly ordered topological space contains
a ccc subspace of size $\aleph_1$.
\end{corollary}

Our goal now is to prove a stronger version of Theorem
\ref{firstcountable} for the class of Hausdorff spaces (see
Theorem \ref{main} below). For this, recall that the
\emph{pointwise compactness type} of a space $X$, denoted by
$\mathrm{pct}(X)$, is the smallest infinite cardinal $\kappa$ for
which $X$ can be covered by compact subsets $K$ with $\chi(K,X)
\leq \kappa$. When $\mathrm{pct}(X)=\aleph_0$ we say that $X$ has
\emph{pointwise countable type}.\\

Fix a Hausdorff topological space $(X,\tau)$ and an elementary
submodel $M$ of $H(\theta)$ (for a large enough regular cardinal
$\theta$) such that $(X,\tau) \in M$. We let
$$\mathcal{K}=\{K\subseteq X: K \mbox{ is compact and } \chi(K,X)
 \leq \mathrm{pct}(X)\}.$$
By definition we have that $\mathcal{K}$ is a cover of $X$ and it
is not hard to see, using the fact that in a compact Hausdorff
space any open set is a union of compact $G_\delta$ subsets, that
$\mathcal{K}$ is a network of $X$ (i.e. any open subset of $X$ is
a union of elements of $\mathcal{K}$). Note also that
$\mathcal{K}$ is closed under finite intersections. Now for each
$p \in X \cap M$ we let $$\Delta_p=\bigcap \{ U \in \tau \cap M :
p \in U \}.$$

Note that by elementarity, since $X$ is Hausdorff, if $p,q \in X
\cap M$ are distinct then $\Delta_p \cap \Delta_q=\emptyset$.
Also, if $A$ is any intersection of open subsets of $X$ which are
in $M$ and $p \in A \cap M$ then clearly $\Delta_p \subseteq A$.
In particular this is true of any $A \in \mathcal{P}(X) \cap M$
for which $\chi(A,X) \subseteq M$.

\begin{lemma}\label{intersection}
If $p \in X \cap M$ and $\mathrm{pct}(X) \subseteq M$ then
$$\Delta_p=\bigcap \{K \in \mathcal{K} \cap M : p \in K \}.$$
\end{lemma}
\begin{proof}
Note that $\psi\chi(K,X) \subseteq M$ for all $K \in \mathcal{K}$,
so $\Delta_p \subseteq \bigcap \{ K \in \mathcal{K} \cap M : p \in
K \}$. On the other hand, by elementary and the fact that
$\mathcal{K}$ is a network of $X$, for each $U \in \tau \cap M$
with $p \in U$ we can find $K_U \in \mathcal{K} \cap M$ such that
$p \in K_U \subseteq U$. But then $\bigcap \{ K \in \mathcal{K}
\cap M : p \in K \} \subseteq \bigcap\{K_U : p \in U \in \tau \cap
M\} \subseteq \bigcap \{ U : p \in U \in \tau \cap M\}= \Delta_p$,
which finishes the proof.
\end{proof}

As we mentioned before, the topology in $X \cap M$ generated by
$\tau \cap M$ is often strictly coarser than the subspace
topology. So given $U \in \tau$ and $p \in U \cap M$ there is no
guarantee that there is a $V \in \tau \cap M$ with $p \in V
\subseteq U$. However we have the following

\begin{lemma}\label{almostbase}
Suppose that $\mathrm{pct}(X) \subseteq M$. For any $U \in \tau$
and $p \in U \cap M$, if $\Delta_p \subseteq U$ then there is a $V
\in \tau \cap M$ such that $p \in V \subseteq U$.
\end{lemma}
\begin{proof}
Let $\mathcal{C}=\{ K \cap (X \setminus U) : K \in \mathcal{K}
\cap M, p \in K \}$. By Lemma \ref{intersection} we have that
$\bigcap \mathcal{C}=\Delta_p \cap (X \setminus U)=\emptyset$. But
$\mathcal{C}$ is a collection of compact subsets of $X$ closed
under finite intersections and therefore $\emptyset \in
\mathcal{C}$, so there is a $K \in \mathcal{K} \cap M$ with $p \in
K \subseteq U$. Since $\chi(K,X) \subseteq M$ by elementarity $K$
has an outer base entirely contained in $M$ and hence there is a
$V \in \tau \cap M$ such that $p \in K \subseteq V \subseteq U$.
\end{proof}

We are ready to prove our main result.

\begin{theorem}\label{main}
Let $X$ be a ccc Hausdorff space and $\kappa$ a cardinal such that
$\mathrm{pct}(X) \leq \kappa \leq |X|$. Then there is a ccc
subspace $Y \subseteq X$ with $|Y|=\kappa$.
\end{theorem}
\begin{proof}
Take $M$ such that $\kappa \subseteq M$ and $|M|=\kappa$. Consider
the set $$\mathcal{C}= \{C \subseteq X: C \mbox{ is closed and } C
\cap \Delta_p \neq \emptyset \mbox{ for all } p \in X \cap M \}$$
ordered by inclusion. By Lemma \ref{intersection} we have that
each $\Delta_p$ is compact which allows us to use Zorn's lemma to
get a minimal element $D$ of $\mathcal{C}$. We prove now that $D$
is ccc.

Fix a collection $\{U_\alpha : \alpha \in \omega_1 \} \subseteq
\tau$ such that $U_\alpha \cap D \neq \emptyset$ for all $\alpha
\in \omega_1$. For a given $\alpha \in \omega_1$ we have that $D
\cap (X \setminus U_\alpha)$ is closed and properly contained in
$D$. Hence by minimality of $D$ there is a $p_\alpha \in X \cap M$
such that $D \cap (X \setminus U_\alpha) \cap
\Delta_{p_\alpha}=\emptyset$ and therefore $\Delta_{p_\alpha}
\subseteq (X \setminus D) \cup U_\alpha$. Using Lemma
\ref{almostbase} we get a $V_\alpha \in \tau \cap M$ such that
$p_\alpha \in V_\alpha \subseteq (X \setminus D) \cup U_\alpha$.
Since $X$ is ccc there are $\alpha, \beta \in \omega_1$ such that
$V_\alpha \cap V_\beta \neq \emptyset$ so by elementary there is a
$q \in V_\alpha \cap V_\beta \cap M$. But then $U_\alpha \cap
U_\beta \cap D \supseteq V_\alpha \cap V_\beta \cap D \supseteq
\Delta_q \cap D \neq \emptyset$, so $\{U_\alpha \cap D : \alpha
\in \omega_1 \}$ is not a cellular family in $D$.

Now for each $p \in X \cap M$ choose a point $y_p \in D \cap
\Delta_p$ and let $Y=\{ y_p : p \in X \cap M \} \subseteq X$.
Since $\kappa \subseteq M$ and $|M|= \kappa$ we know that $|X \cap
M|=\kappa$. Moreover all the $\Delta_p$'s are disjoint so all the
$y_p's$ are different and hence $|Y|=\kappa$. Finally note that
$\overline{Y} \in \mathcal{C}$ so by minimality of $D$ we have
$\overline{Y}=D$ and therefore $Y$ is ccc.
\end{proof}

Since compact (and even locally compact) spaces have pointwise
countable type, we get

\begin{corollary}
Every uncountable ccc (locally) compact Hausdorff space contains a
ccc subspace of size $\aleph_1$.
\end{corollary}

Note that the Hausdorff condition was not needed in Theorem
\ref{firstcountable}. However, the following example shows that
Theorem \ref{main} fails for $T_1$ spaces even in the compact
case.

\begin{example}
For each regular cardinal $\kappa$, there is a ccc compact $T_1$
space of size $\kappa$ with no uncountable ccc subspace of size
less than $\kappa$.
\end{example}
\begin{proof}
Let $\tau_o$ be the usual order topology on the ordinal
$\kappa+1$. Let $X=\kappa$ with the topology $\tau=\{U \cap
\kappa: \kappa \in U \in \tau_o \}$. The space $X$ is clearly
$T_1$ and it is ccc because any two nonempty open subsets of $X$
intersect. Note that for any $\alpha \in \kappa$, the subspace
$\alpha \subseteq X$ has the usual order topology, so any open
subset of $X$ covers all but a compact subspace of $X$. Hence $X$
is compact. Since, by regularity of $\kappa$, any subspace $Y$ of
$X$ of size less than $\kappa$ is contained in some $\alpha \in
\kappa$, it follows that $Y$ has a topology finner than the one
induced by its order $(Y, \in)$ and therefore $Y$ is not ccc
unless $Y$ is countable.
\end{proof}

\begin{remark}
The previous are also examples of compact $T_1$ spaces with no
Lindel\"of subspace of size $\aleph_1$, which is of independent
interest.
\end{remark}

The space given in Example \ref{hausdorffexample} has
pseudocharacter $\kappa$ so taking $\kappa=\aleph_1$ we see that
$\chi(X) \leq \aleph_1$ in Theorem \ref{firstcountable} cannot be
weakened to $\psi\chi(X) \leq \aleph_1$ even for Hausdorff spaces.
Although the space in Example \ref{hausdorffexample} is not
regular, the following example shows that non-regularity is not
the problem.

\begin{example}
For each regular cardinal $\kappa$, there is a ccc Tychonoff space
$X$ of size $\kappa$ and countable pseudocharacter with no
uncountable ccc subspace of size less than $\kappa$.
\end{example}
\begin{proof}
Fix a function $\varphi: \kappa \to [\kappa]^{<\omega}$ such that
$|\varphi^{-1}(\{s\})|=\kappa$ for every $s \in
[\kappa]^{<\omega}$. Given $x \in 2^\kappa$ and $\alpha \in
\kappa$ we define $s(\alpha,x)=\{\xi \leq \alpha : x(\xi)=1\}$.
For $x \in 2^\kappa$ we let $G(x)=\{\alpha \in \kappa: x(\alpha)=1
\mbox{ and } s(\alpha,x) \mbox{ is finite}\}$. Now we let $$X=\{x
\in 2^\kappa : \exists \alpha \in G(x), \forall \beta>\alpha \
[x(\beta)=1 \leftrightarrow \varphi(\beta)=s(\alpha,x)]  \}$$ with
the subspace topology inherited from $2^\kappa$ (so $X$ is
Tychonoff).

First we show that $X$ is dense in $2^\kappa$ and therefore $X$ is
ccc being a dense subspace of a ccc space. For this, fix a finite
partial function $\sigma:\kappa \to 2$ and denote by $[\sigma]$
the basic clopen subset of $2^\kappa$ determined by $\sigma$. Pick
$\alpha \in \kappa$ with $\mathrm{dom}(\sigma)<\alpha$ and define
$x \in 2^\kappa$ by $\sigma \subseteq x$ and for $\beta \notin
\mathrm{dom}(\sigma)$:
$$x(\beta)=1 \iff \beta=\alpha \mbox{ or } \beta>\alpha \land
\varphi(\beta)=\{\alpha\}\cup \sigma^{-1}(\{1\}).$$ Then $x \in
[\sigma]$ and $s(\alpha,x)=\{\alpha\}\cup \sigma^{-1}(\{1\})$
hence $\alpha \in G(x)$ and $x \in [\sigma] \cap X$. Thus $X$ is
dense in $2^\kappa$.

Note that for each $x \in X$ we can choose $\alpha_x \in \kappa$
witnessing the fact that $x \in X$ and the map $x \mapsto
(\alpha_x, s(\alpha_x,x))$ is an injection from $X$ into $\kappa
\times [\kappa]^{<\omega}$ so $|X| \leq \kappa$. Moreover for each
$\alpha \in \kappa$ we can define $x_\alpha \in X$ by
$x_\alpha(\beta)=0$ if $\beta < \alpha$, $x_\alpha(\alpha)=1$ and
for $\beta>\alpha$, $x_\alpha(\beta)=1 \leftrightarrow
\varphi(\beta)=\{\alpha\}$. It is clear that all the $x_\alpha$'s
are distinct so $|X| \geq \kappa$ and hence $|X|=\kappa$.

To see that $\psi\chi(X)=\aleph_0$, fix $p \in X$ and let $A
\subseteq \kappa$ consist of the first $\omega$-many elements of
$\kappa$ for which $p$ takes the value $1$ (note that any element
of $X$ takes the value $1$ infinitely many times). Using the fact
that the map $x \mapsto (\alpha_x, s(\alpha_x,x))$ defined in the
previous paragraph is injective, it is easy to see that $p$ is the
only element of $X$ that takes value $1$ in all the ordinals in
$A$. But this condition defines a $G_\delta$ in $X$ so we are
done.

Finally, let $Y$ be any subspace of $X$ with $|Y|< \kappa$. We
will show that $Y$ is discrete and hence not ccc unless countable.
For each $y \in Y$ let $\alpha_y \in \kappa$ be a witness of the
fact that $y \in X$ and take $\delta=\sup\{\alpha_y: y \in Y\}$.
By regularity of $\kappa$, $\delta \in \kappa$. Now fix $y \in Y$
and choose $\beta \in \kappa$ such that $\beta > \delta$ and
$\varphi(\beta)=s(\alpha_y,y)$. This can be done since
$s(\alpha_y,y)$ is finite and therefore
$\varphi^{-1}(\{s(\alpha_y,y)\})$ is cofinal in $\kappa$. Let
$\sigma=\{(\beta,1)\}$ and note that, since $\beta > \delta \geq
\alpha_y$ and $\varphi(\beta)=s(\alpha_y,y)$, we have $y(\beta)=1$
and hence $y \in [\sigma] \cap Y$. On the other hand if $z \in
[\sigma] \cap Y$ then $z(\beta)=1$ and therefore, since $\beta >
\delta \geq \alpha_z$,
$s(\alpha_z,z)=\varphi(\beta)=s(\alpha_y,y)$. In particular
$\alpha_z=\max s(\alpha_z,z)=\max s(\alpha_y,y)=\alpha_y$. This
implies that $z \restriction (\alpha_z+1)=y \restriction
(\alpha_y+1)$. Since for all $\xi > \alpha_z=\alpha_y$ we have
$z(\xi)=1 \leftrightarrow \varphi(\xi)=s(\alpha_z,z)
\leftrightarrow \varphi(\xi)=s(\alpha_y,y) \leftrightarrow
y(\xi)=1$, it follows that $z=y$. Thus $[\sigma] \cap Y=\{y\}$,
and $Y$ is discrete.
\end{proof}

Using this example we can easily set up a situation where we have
a ccc Tychonoff space $X$ which contains a ccc subspace of size
$\aleph_1$ but $X$ is such that for any elementary submodel $M$
with $\omega_1 \cup \{X\} \subseteq M$ and $|M|=\aleph_1$, the
subspace $X \cap M$ is not ccc. For instance, just take the
disjoint union $X=Y \cup Z$ where $Y$ is any subspace of
$\mathbb{R}$ of size $\aleph_1$ and $Z$ is the space in the
previous example with $\kappa=\aleph_2$. Then $X \cap M = Y \cup
(Z\cap M)$ which is not a ccc space since $Z \cap M$ is an open
subspace which is not ccc. However we don't know the answer to the
following

\begin{question}
Suppose $X$ is a compact Hausdorff space and $M$ is an elementary
submodel with $\omega_1 \cup \{X\} \subseteq M$ and
$|M|=\aleph_1$. Is it true that the subspace $X \cap M$ is ccc?
\end{question}


\begin{thebibliography}{9}

\bibitem{BT}
J.E. Baumgartner, F.D. Tall, Reflecting Lindel\"ofness,
\textit{Topology Appl.} 122 (2002) 35-49.

\bibitem{D}
A. Dow, An introduction to applications of elementary submodels to
topology, \textit{Topology Proc.} 13 (1988) 17-72.

\bibitem{HJ}
A. Hajnal, I. Juh\'asz, Remarks on the cardinality of compact
spaces and their Lindel\"of subspaces, \textit{Proc. Amer. Math.
Soc.} 59 (1976) 146-148.

\bibitem{JT}
L.R. Junqueira, F.D. Tall, The topology of elementary submodels,
\textit{Topology Appl.} 82 (1998) 239-266.

\bibitem{KT}
P. Koszmider, F.D. Tall, A Lindel\"of space with no Lindel\"of
subspace of size $\aleph_1$, \textit{Proc. Amer. Math. Soc.} 130
(2002) 2777-2787.

\bibitem{S}
M. Scheepers, Measurable cardinals and the cardinality of
Lindel\"of spaces, \textit{Topology Appl.} 157 (2010) 1651-1657.


\end{thebibliography}
\end{document}